\documentclass[a4paper, 12pt]{amsart}
\usepackage{amssymb, amsmath, amsthm, color}
\usepackage[colorlinks=true, breaklinks=true, linkcolor=black, citecolor=blue, urlcolor=red]{hyperref} 
\usepackage[francais,english]{babel}
\usepackage{tikz-cd}

\newtheorem{theorem}{Theorem}[section]

\newtheorem{corollary}[theorem]{Corollary}
\newtheorem{proposition}[theorem]{Proposition}

\newtheorem{remark}[theorem]{Remark}

\newcommand{\onto}{\twoheadrightarrow}
\newcommand{\into}{\hookrightarrow}

\newcommand{\C}{\mathbf C}
\newcommand{\cC}{\mathcal C}

\DeclareMathOperator{\id}{id}

\newcommand{\N}{\mathbb N}

\newcommand{\cO}{\mathcal O}

\newcommand{\ot}{\otimes}
\newcommand{\cP}{\mathcal P}

\DeclareMathOperator{\Rep}{Rep}

\newcommand{\cS}{\mathcal S}
\DeclareMathOperator{\State}{State}

\newcommand{\cT}{\mathcal T}

\newcommand{\varep}{\varepsilon}
\newcommand{\wh}{\widehat}

\begin{document}

\title{Representations of the infinite Cuntz algebra}

\author{Arnaud Brothier and David E.~Evans}
\address{Arnaud Brothier\\	University of Trieste, Department of Mathematics, Informatics and Geosciences, Section Mathematics, via Valerio 12/1, 34127, Trieste, Italy and School of Mathematics and Statistics, University of New South Wales, Sydney NSW 2052, Australia}
\email{arnaud.brothier@gmail.com\endgraf
		\url{https://sites.google.com/site/arnaudbrothier/}}
	\address{David E.~Evans\\ School of Mathematics, Cardiff University, Senghennydd Road, Cardiff CF24 4AG, Wales, U.K.}
	\email{EvansDE@cf.ac.uk}

\maketitle

\begin{abstract}
We define an intrinsic unital completely positive (ucp) map from the Cuntz algebra to the Pythagorean with finitely many generators. This permits to lift states and cyclic representations. 
This process restricts to the Cuntz-Toeplitz algebras and the contraction algebras. Taking inductive limits we obtain an effective machinery for constructing states on the infinite Cuntz algebra $\cO_\infty$.
Combining this approach with previous techniques of the first author we obtain large families of pairwise inequivalent irreducible representations of $\cO_\infty$: one family being in an obvious bijection with the unit sphere of Hilbert's space.
\end{abstract}

\section*{Introduction}

Cuntz algebras $\cO_n$ are fundamental objects in the study of C*-algebras \cite{Dixmier64, Cuntz77}, see also \cite{Evans80}. They were discovered by Dixmier and then got deeply studied by Cuntz.
They were the first examples of separable, nuclear, purely infinite and simple C*-algebras. 
Remarkably, $\cO_2$ and $\cO_\infty$ play the role of zero and unity, respectively, when considering their tensor product with a large class of C*-algebras \cite{Kirchberg-Phillips00}. 
This makes these two algebras essential in the classification program of C*-algebras.

Constructing representations of them is rather easy but checking irreducibility and unitary conjugacy of given representations of them are difficult tasks.
Cuntz algebras are simple and of non type I. 
By Glimm theorem we deduce that they have a pathological spectrum (i.e.~the unitary classes of irreducible representations) \cite{Glimm61}.

In \cite{Brothier-Jones19} was defined the Pythagorean algebra $\cP_n$ for finite $n$ consisting in taking the Cuntz algebra and removing the isometry condition on the generators.
A remarkable explicit functorial process $\Pi:\Rep(\cP_n)\to\Rep(\cO_n)$ permits to lift representations of $\cP_n$ into representations of $\cO_n$.
All representations of $\cO_n$ arise in that manner (i.e.~$\Pi$ is essentially surjective).
However, the functor $\Pi$ typically does \emph{not} preserve irreducibility nor unitary classes.
It was later discovered, in a series of articles of Wijesena and the first author, a stronger notion of irreducibility for representations of $\cP_n$ so that their images by $\Pi$ remain irreducible and for which unitary conjugacy is fully preserved \cite{Brothier-Wijesena25,Brothier-Wijesena26, Brothier-Wijesena24}. 
These have permitted to describe a large part of the spectrum of $\cO_n$ as a disjoint union of smooth manifolds.
These techniques were very successful. 
Though, they were based on graphical calculus of categories of locally finite forests.
These techniques and constructions do not seem to adapt for the infinite Cuntz algebra $\cO_\infty$.

We develop in this article a new approach for constructing and classifying representations of the Cuntz algebras using unital completely positive (ucp) maps and states. 
By definition, there is an obvious quotient map $q_n:\cP_n\onto\cO_n$.
We construct a section $\varphi_n$ of $q_n$ which is a ucp map.
It is in fact the unique such ucp section.
This permits to transform states of $\cP_n$ to states of $\cO_n.$ 
Using the Gelfand-Naimark-Segal (GNS) construction we deduce a new process for lifting (cyclic) representations of $\cP_n$ to representations of $\cO_n$.
We observe that this process based on states (mostly) coincides with the previous functor $\Pi:\Rep(\cP_n)\to\Rep(\cO_n)$ of \cite{Brothier-Jones19}.

Our strategy continues to work for the Cuntz-Toeplitz algebras 
$$\cT_n:=C^*\langle t_1,\cdots,t_n| t_j^*t_i=\delta_{i,j}, \sum_i t_i^*t_i\leq 1\rangle$$ and the contraction algebras $$\cC_n:=C^*\langle c_1,\cdots,c_n| \sum_i c_i^*c_i\leq 1\rangle.$$
These C*-algebras can be interpreted as the subalgebras of $\cO_{n+1}$ and $\cP_{n+1}$ generated by the $n$ first generators (but not the last one), respectively. 
Then our ucp map $\varphi_n:\cO_{n+1}\to\cP_{n+1}$ restricts to a ucp map $\cT_n\to \cC_n.$
We may now lift representations of $\cC_n$ into representations of $\cT_n$.
Additionally, there are obvious embeddings $\cC_n\into \cC_{n+1}$ and $\cT_n\into\cT_{n+1}$ (unlike in the Pythagorean and Cuntz algebra cases) that are compatible with our ucp maps.
Taking direct limits we obtain a ucp map
$$\varphi_\infty:\cO_\infty\to\cP_\infty$$
where $\cO_\infty$ and $\cP_\infty$ have infinitely many generators satisfying a key inequality (rather than equality) in their presentation.
Any state $\sigma$ of $\cP_\infty$ defines a state $\sigma\circ \varphi_\infty$ on $\cO_\infty$.
Then, the GNS construction produces a cyclic representation of $\cO_\infty$.
All cyclic representations of $\cO_\infty$ arise in that manner (and in particular all irreducible representations of $\cO_\infty$).

Finally, by mimicking strategies developed by Wijesena and the first author, we obtain a plethora of irreducible representations of $\cO_\infty$ that are unitarily inequivalent. 
In particular, each point $z=(z_n)_{n\geq 0}$ of the unit sphere of Hilbert's space $\ell^2(\N)$ defines an irreducible representation $\pi_z$ of $\cO_\infty$ so that $\pi_z\simeq\pi_w$ if and only if $z=w.$

{\bf AI statement.} 
The basic free versions of ChatGPT, Mistral and Claude have been used for finding and checking references and additionally for correcting typos. All strategies and proofs have been entirely discovered and written down by the authors only. 

{\bf Acknowledgements.} This work initiated from a lively discussion of the authors together with Corey Jones at the Operator Algebras and Mathematical Physics workshop hosted at ICMS, Bayes Center, Edinburgh in the fall of 2025. We are very grateful to the organisers and the sponsors of this event. We warmly thank Corey for throwing pertinent ideas and for his communicative enthusiasm toward this project. 
An Emeritus Leverhulme Fellowship has supported a visit of DE in May 2026 in Trieste. We thank the sponsor and the hospitality of the University of Trieste.
AB thanks Jamie Gabe for enlightening discussions on the Cuntz algebras.

\section{Definitions and notations}
All Hilbert spaces and C*-algebras considered are separable.
A morphism, map or arrow between C*-algebras means a continuous unital *-algebra morphism.

Consider a natural number $n\geq 1$.
Define the $n$th Cuntz algebra $\cO_n$ to be the universal C*-algebra with presentation
$$\cO_n:=C^*\langle s_1,\cdots,s_n| \sum_{i=1}^n s_i^*s_i=1, s_j^*s_j=1, \forall 1\leq j\leq n\rangle.$$
This means that for any Hilbert space $H$ and family of isometries $S_1,\cdots,S_n$ satisfying $\sum_i S_i^*S_i=\id_H$ there is a unique representation of $\cO_n$ sending $s_i$ to $S_i.$ 
This C*-algebra $\cO_n$ is simple: all nonzero representations are faithful. 
If $\mu=\mu_1\cdots\mu_m\in \{1,\cdots,n\}^*$ is a finite word with letters in $\{1,\cdots,n\}$ we write $s_\mu$ for $s_{\mu_1}\cdots s_{\mu_m}.$
It is standard that the closed linear span of the $s_\mu s_\nu^*$ with $\mu,\nu\in\{1,\cdots,n\}^*$ is equal to $\cO_n$ \cite{Cuntz77}.
We will often write $(S_i,H)$ to express a representation of $\cO_n$ and write $S_\mu$ for $S_{\mu_1}\cdots S_{\mu_m}.$

Define the $n$th Pythagorean algebra 
$$\cP_n:=C^*\langle a_1,\cdots,a_n| \sum_{i=1}^na_i^*a_i=1\rangle.$$
Observe that the formula $a_i\mapsto s_i^*$ defines a unique (surjective) morphism $$q_n:\cP_n\onto\cO_n.$$ 
Hence, $\cO_n$ is a quotient of $\cP_n.$
If $\mu=\mu_1\cdots\mu_m\in \{1,\cdots,n\}^*$, then we write $a_\mu$ for $a_{\mu_m}\cdots a_{\mu_1}.$
Hence, we have reversed the order. This makes the formula easier to write.
In particular, the quotient map $\cP_n\onto \cO_n$ sends $a_\mu$ to $s_\mu^*.$
If $K$ is a Hilbert space and $A_1,\cdots,A_n\in B(K)$ are operators satisfying $\sum_{i=1}^nA_i^*A_i=\id_K$, then there exists a unique representation of $\cP_n$ on $K$ sending $a_i$ to $A_i.$
We may then express representations of $\cP_n$ as tuples $(A_i,H).$

Define the $n$th Cuntz-Toeplitz algebra to be
$$\cT_n:=C^*\langle s_1,\cdots,s_n\rangle \subset \cO_{n+1}$$
the C*-subalgebra generated by the first $n$ generators of $\cO_{n+1}$.
A generalisation of Coburn theorem proved by Cuntz shows that $\cT_n$ has for presentation
$$C^*\langle t_1,\cdots, t_n| \sum_{i=1}^nt_it_i^*\leq 1, t_j^*t_j=1,\forall 1\leq j\leq n\rangle$$
which can be written in a more condensed manner as
$$C^*\langle t_1,\cdots,t_n| t_j^*t_i=\delta_{i,j}, \forall 1\leq i,j\leq n\rangle.$$
There is an obvious quotient map $\cT_n\onto \cO_n, t_i\mapsto s_i$ whose kernel is the ideal generated by $1-\sum_{i=1}t_it_i^*$. This latter is isomorphic to the C*-algebra of the compact operators.
A representation of $\cT_n$ consists in $n$ isometries with orthogonal ranges. If the direct sum of these ranges is proper, then the representation is automatically faithful and otherwise defines a representation of $\cO_n.$
Cuntz-Toeplitz algebras have the advantage to have obvious embeddings $\cT_n\into \cT_{n+1}$ while the Cuntz algebras don't.

Define the $n$th contraction algebra to be
$$\cC_n:=C^*\langle a_1,\cdots,a_n\rangle \subset \cP_{n+1}.$$
It is straightforward to prove that $\cC_n$ admits the following presentation
$$C^*\langle c_1,\cdots,c_n| \sum_{i=1}^n c_i^*c_i\leq 1\rangle.$$
We have obvious quotient maps $\cC_n\onto\cP_n,c_i\mapsto a_i$ and $\cC_n\onto \cT_n, c_i\mapsto t_i^*$.
Just like for the $\cT_n$'s, we have obvious embeddings $\cC_n\into \cC_{n+1}$ while we don't have such embeddings between the Pythagorean algebras.

The infinite Cuntz algebra $\cO_\infty$ is the universal C*-algebra of countably many isometries $s_1,s_2,\cdots$ with orthogonal ranges.
This is again a simple C*-algebra. 
Hence, a countably infinite family of isometries with orthogonal ranges defines a faithful representation of $\cO_\infty.$
The direct sum (i.e.~the completion of the infinite algebraic direct sum) of the range of these isometries may be the whole space or not. 
Knowing this we deduce that the inductive limit of the Cuntz-Toeplitz algebra $\cT_n$ is isomorphic to $\cO_\infty$:
$$\varinjlim_n \cT_n\simeq \cO_\infty.$$

For commodity we define the infinite Pythagorean algebra $\cP_\infty$ to be the universal C*-algebra having for generators $a_1,a_2,\cdots$ and satisfying that for any finite subset $F\subset \N$ we have $\sum_{i\in F} a_i^*a_i\leq 1.$
This defines a C*-algebra isomorphic to the direct limit of the contraction algebras:
$$\varinjlim_n \cC_n\simeq \cP_\infty.$$

Observe that when $n=1$, then $\cO_1,\cT_1,\cC_1$ are the universal C*-algebras generated by a unitary, an isometry, and a contraction, respectively, and $\cT_1\simeq \cP_1.$

\section{A unital completely positive map}
Fix $n\geq 1.$
Take a C*-algebra $A$ that contains copies of $\cO_n$ and $\cP_n$ (for instance $\cO_n\ot \cP_n$) and consider the element 
$$y:=\sum_{i=1}^n s_ia_i\in A.$$
This is an isometry satisfying $y^*s_j=a_j^*$ for $1\leq j\leq n$. 
Dilate the co-isometry $y^*$ into the unitary
$$u:=\begin{pmatrix}
    y^* & 0 \\
    p & y
\end{pmatrix} \in M_2(A) \text{ where } p:=1-yy^*.$$
Since $u$ is a unitary, we have that the mapping 
$$s_i\mapsto S_i:= u \begin{pmatrix}
    s_i & 0 \\ 0 & s_i
\end{pmatrix}$$
defines a morphism $\pi:\cO_n\to M_2(A)$ by universality of $\cO_n.$
We may now define our unital completely positive map (ucp map) 
$$\varphi:\cO_n\to A,\ x\mapsto E\ \pi(x)\ E \text{ where } E=\begin{pmatrix}
    1& 0\\ 0& 0
\end{pmatrix}.$$

\begin{proposition}
The ucp map $\varphi$ is valued in $\cP_n$ and satisfies that 
$$\varphi(s_\mu s_\nu^*)=a_\mu^* a_\nu \text{ for all } \mu,\nu\in\{1,\cdots,n\}^*.$$
In particular, $\varphi$ is a section of the quotient map $q_n:\cP_n\onto \cO_n.$
\end{proposition}

\begin{proof}
Fix some words $\mu=\mu_1\cdots \mu_m$ and $\nu=\nu_1\cdots \nu_k$.
By definition $S_i=\pi(s_i)$ is a lower-triangular matrix with top-left coefficient equal to $y^*s_i$ which is $a_i^*.$
Hence, $S_\mu$ is a lower-triangular matrix with top-left coefficient $a_{\mu_1}^*\cdots a_{\mu_m}^*=(a_{\mu_m}\cdots a_{\mu_1})^*=a_\mu^*$ according to our conventions.
Therefore, $S_\nu^*$ is an upper-triangular matrix with top-left coefficient $a_\nu.$
We deduce that the top-left coefficient of $S_\mu S_\nu^*$ is $a_\mu^*a_\nu$ and thus $\varphi(s_\mu s_\nu^*)=a_\mu^*a_\nu.$
This proves the formula.
Since the $s_\mu s_\nu^*$ form a total family in $\cO_n$ we deduce that $\varphi$ is valued in $\cP_n.$
Finally, observe that $q_n\circ \varphi(s_\mu s_\nu^*)=s_\mu s_\nu^*$ implying by density that $q_n\circ \varphi$ is the identity of $\cO_n.$
\end{proof}

\begin{remark}
One can prove that there is only one ucp map $\psi:\cO_n\to\cP_n$ satisfying $\psi(s_i)=a_i^*.$
Hence, such a ucp map $\psi$ must be equal to $\varphi$ and thus be a section of the quotient map $q_n:\cP_n\onto\cO_n.$
This exhibits a certain canonical property of our ucp map $\varphi:\cO_n\to\cP_n$.
\end{remark}

\section{Lifting states and representations}

We have proved that the section $\varphi:\cO_n\to \cP_n, s_\mu s_\nu^*\mapsto a_\mu^*a_\nu$ of $q_n:\cP_n\to \cO_n$ is ucp.
Therefore, it defines a map of state spaces:
$$\State(\varphi):\State(\cP_n)\to \State(\cO_n), f\mapsto f\circ \varphi.$$
This map must be surjective since $q_n\circ \varphi=\id_{\cO_n}.$
Hence, all states of $\cO_n$ can be realised using states of $\cP_n.$

Let us use $\varphi$ to lift representations of $\cP_n$ into representations of $\cO_n.$
Hence, take $(\pi,K)$ to be a representation of $\cP_n$. Consider a unit vector $\xi$ of $K$ and form the vector state
$$f:\cP_n\to\C, x\mapsto \langle \pi(x)\xi,\xi\rangle.$$
Compose now with $\varphi$ obtaining $\wh f:=f\circ\varphi$ a state of $\cO_n$.
We apply the Gelfand-Naimark-Segal (GNS) construction to $\wh f$ obtaining a triple $(\wh \pi, \wh K,\wh\xi)$ where $(\wh\pi,\wh K)$ is a representation of $\cO_n$, $\wh\xi$ is a cyclic unit vector of $\wh K$ and 
$$\langle \wh\pi(z) \wh \xi,\wh \xi\rangle =\wh f(z) \text{ for all } z\in\cO_n.$$
Using again that $\varphi$ is a section of $q_n$ we deduce that $(\wh \pi,\wh K)$ can be \emph{any} representation of $\cO_n$ admitting a cyclic vector.
In particular, one can construct \emph{every} irreducible representation of $\cO_n$ in that manner.

In \cite{Brothier-Jones19}, a process was defined to lift representations of $\cP_n$ to representations of $\cO_n.$
This process was later extensively studied (in the $n=2$ case) in \cite{Brothier-Wijesena25,Brothier-Wijesena26,Brothier-Wijesena24}. 
This process defines the so-called Pythagorean functor 
$$\Pi:\Rep(\cP_n)\to\Rep(\cO_n)$$
between the representation category satisfying that the functor $\Pi\circ U$ is naturally isomorphic to the identity functor of $\Rep(\cO_n)$ where 
$$U:\Rep(\cO_n)\to \Rep(\cP_n)$$
is the forgetful functor constructed via the quotient map $q_n:\cP_n\onto\cO_n.$
Hence, $\Pi$ can be interpreted as a functorial section of $q_n$. 
We prove below that $\Pi$ is closely related to the process $(\pi,K,\xi)\mapsto (\wh\pi,\wh K,\wh\xi)$ described above.

To explain the connection we recall some facts on the functor $\Pi$.
Fix a representation $(\pi,K)$ of $\cP_n$ and write $A_i$ for $\pi(a_i).$
Consider now $\ell^2(\{1,\cdots,n\}^*)\otimes K$ the Hilbert space with family of total vectors $\delta_\mu\otimes \eta$ with $\mu$ a finite word in the letters $1,\cdots,n$ and $\eta\in K.$
Mod out this space by the equivalence relation $\sim$ generated by
$$\delta_\mu\otimes \eta\sim \sum_{i=1}^n \delta_{\mu\cdot i}\ot A_i(\eta).$$
Since $(A_i,K)$ is a representation of $\cP_n$ we have that $\eta\mapsto \oplus_i A_i(\eta)$ is isometric. 
Hence, $\delta_\mu\otimes \eta$ and $\sum_{i=1}^n \delta_{\mu\cdot i}\ot A_i(\eta)$ have same norm.
The quotient provides a Hilbert space denoted $\Pi(K)$.
We write $[\mu,\eta]$ for the class of $\delta_\mu\ot \eta$ in $\Pi(K)$ which has norm $\|\eta\|.$
Set $S_i(\delta_\mu\ot \eta):=\delta_{i\cdot \mu}\ot \eta$.
This defines a representation of $\cO_n$ on $\Pi(K).$
Additionally, we have an isometric embedding 
$$j_K:K\into \Pi(K), \eta\mapsto [\emptyset,\eta]$$
satisfying that 
$$j_K(A_i\eta)=S_i^*j_K(\eta).$$
Hence, the adjoint of the $S_i$'s can be interpreted as dilations of the $A_i$'s.

Take now a unit vector $\xi$ of $K$ and consider its image $j_K(\xi)$ inside $\Pi(K).$
Let 
$$f:\cP_n\to\C, x\mapsto \langle x\cdot\xi,\xi\rangle$$ be the vector state associated to $(A_i,K,\xi)$ and let 
$$g:\cO_n\to\C, z\mapsto \langle z\cdot j_K(\xi),j_K(\xi)\rangle$$ be the vector state of $\cO_n$ associated to the triple $(S_i,\Pi(K),j_K(\xi)).$
Observe that if $\mu,\nu\in\{1,\cdots,n\}^*$, then 
\begin{align*}
g(s_\mu s_\nu^*)& =\langle S_\mu S_\nu^* j_K(\xi),j_K(\xi)\rangle  = \langle S_\nu^* j_K(\xi), S_\mu^* j_K(\xi)\rangle\\
& = \langle j_K(A_\nu \xi) , j_K(A_\mu \xi)\rangle
 = \langle A_\nu \xi, A_\mu \xi\rangle\\
& = \langle A_\mu^*A_\nu \xi, \xi\rangle = f(a_\mu^* a_\nu).
\end{align*}
By density we deduce the following.

\begin{proposition}
Let $f$ be a state of $\cP_n$ and let $(A_i,K,\xi)$ be the associated GNS triple.
Consider now $(S_i,\Pi(K),j_K(\xi))$ as constructed above and write $g$ for the associated vector state.
We have that $g=f\circ \varphi.$
\end{proposition}

We warn the reader that in general, when $(A_i,K,\xi)$ is a representation of $\cP_n$ with cyclic vector $\xi$, then $j_K(\xi)$ may not be cyclic for the representation $\Pi(K)$ of $\cO_n$.
Moreover, there are examples of (finite dimensional) irreducible representations $K$ of $\cP_n$ satisfying that $\Pi(K)$ is reducible.
This shows that $\State(\varphi)$ may send pure states into non-pure states.
We refer to \cite{Brothier-Wijesena24} and \cite{Brothier-Sims-Wijesena26} for details and explicit examples.

Nonetheless we have the following striking result.
\begin{theorem}\cite{Brothier-Wijesena24}
Let $K$ be a finite dimensional representation of $\cP_n$ so that $K$ does not contain any nonzero proper Hilbert subspace $K_0$ satisfying $a_i\cdot K_0\subset K_0$ for all $1\leq i\leq n$.
Then $\Pi(K)$ is an irreducible representation of $\cO_n$.

If $K,K'$ are two such representations (finite dimensional not containing any nonzero proper subspace closed under the $a_i$'s), then $\Pi(K)\simeq \Pi(K')$ if and only if $K\simeq K'$.    
\end{theorem}
We recommend a more readable proof in \cite{Brothier-Sims-Wijesena26}.

\section{The Toeplitz-Cuntz algebra case}
We will now consider several indices $n$. To make the text clearer we now add subscripts $n$ to various symbols.
Recall that we have the inclusions
\begin{equation}\label{eq:inclusion}\cC_n\subset \cP_{n+1} \text{ and } \cT_n\subset \cO_{n+1}\end{equation}
where $\cC_n,\cT_n$ are generated by the first $n$ generators of $\cP_{n+1},\cO_{n+1}$,  respectively.
Moreover, we have the quotient map $q_{n+1}:\cP_{n+1}\onto\cO_{n+1}$ and the ucp map $\varphi_{n+1}:\cO_{n+1}\to\cP_{n+1}.$
These two maps are compatible with the inclusions of \eqref{eq:inclusion}.
Hence, $\varphi_{n+1}$ restricts to a ucp map $\psi_n:\cT_n\to\cC_n$, $q_{n+1}$ restricts into a quotient map $p_n:\cC_n\onto\cT_n$ so that $p_n\circ\psi_n=\id_{\cT_n}$.
Therefore, $\psi_n$ provides a surjective map between the state spaces of $\cC_n$ and $\cT_n$.
In particular, we can lift (cyclic) representations of $\cC_n$ into representations of $\cT_n$ as previously explained in the $\cP_n-\cO_n$ case.
All cyclic representations of $\cT_n$ can be obtained in that manner and in particular all irreducible representations of $\cT_n$ can be built using representations of $\cC_n.$
However, we do not know how to characterise which representations of $\cC_n$ yield \emph{irreducible} representations of $\cT_n$ and how to classify them.

\section{The infinite Cuntz algebra case}

We have obvious embedding $\cC_{n}\into \cC_{n+1}$ and $\cT_{n}\into \cT_{n+1}$ and we may now consider the $\cC_n$'s and the $\cT_n$'s as two increasing towers of C*-algebras.
The maps $p_n:\cC_n\onto \cT_n$ and $\psi_n:\cT_n\to\cC_n$ (defined in the previous short section) are again compatible with these two towers (i.e.~$p_{n+1}\vert_{\cC_n}=p_n$ and $\psi_{n+1}\vert_{\cT_n}=\psi_n$).
We may then take direct limits of C*-algebras and maps.
Recall that $\varinjlim_n\cC_n\simeq\cP_\infty$ and $\varinjlim_n \cT_n\simeq\cO_\infty.$
Define 
$$q_\infty=\varinjlim_n(p_n:\cC_n\onto\cT_n)$$
which provides the quotient map
$$\cP_\infty\onto\cO_\infty, a_i\mapsto s_i^*.$$
Similarly, define
$$\varphi_\infty=\varinjlim_n(\psi_n:\cT_n\to \cC_n)$$
which provides the ucp map 
$$\cO_\infty\to\cP_\infty$$
satisfying $$\varphi_\infty(s_\mu s_\nu^*)=a_\mu^*a_\nu \text{ for all words } \mu,\nu\in\N^*.$$
The ucp map $\varphi_\infty$ is manifestly a section of the quotient map $q_\infty:\cP_\infty\onto\cO_\infty.$
We deduce a surjective state map 
$$\State(\varphi_\infty):\State(\cP_\infty)\onto\State(\cO_\infty)$$
and thus a way to lift cyclic representations of $\cP_\infty$ into representations of $\cO_\infty$.

We will now use this strategy to construct large families of pairwise inequivalent irreducible representations of $\cO_\infty$.
The construction uses the approach developed in this present article but the proof of irreducibility and the classification uses ideas developed in \cite{Brothier-Wijesena24} and simplified in \cite{Brothier-Sims-Wijesena26}. 
In order to apply these techniques we will demand that the representation $(A_i,K)$ of $\cP_\infty$ satisfies the following two axioms:
\begin{enumerate}
    \item $\dim(K)<\infty$ and
    \item the norm limit of the finite sums $\sum_{i\in F} A_i^*A_i$ is equal to the identity $\id_K$ (we want an equality not an inequality).
\end{enumerate}
Under these two assumptions we will be able to address the question of irreducibility and unitary conjugacy for the resulting representations of $\cO_\infty.$

We recall the construction of representations.
Take $(\pi,K,\Omega)$ to be a cyclic representation of $\cP_\infty$ with fixed unit cyclic vector $\Omega$.
Let $\rho:\cP_\infty\to\cO_\infty$ be the associated vector state.
Compose with $\varphi_\infty$ obtaining a state of $\cO_\infty$ and take the GNS triple $(\wh\pi,\wh K,\wh\Omega)$.
We now have a representation $\wh K$ of $\cO_\infty.$

\begin{theorem}
Let $(A_i,K)$ be a finite dimensional representation satisfying that $\sum_{i=1}^\infty A_i^*A_i$ converges to the identity of $K$ and such that if $K_0\subset K$ is a vector subspace satisfying $A_i(K_0)\subset K_0$ for all $i\geq 1$, then $K_0$ is either $\{0\}$ or $K$.
Let $\Omega$ be any unit vector of $K$. 
Apply to $(A_i,K,\Omega)$ the construction of above obtaining a triple $(S_i,\wh K,\wh \Omega)$ where the $S_i$'s denote the representations of $\cO_\infty.$
The following assertions are true.
\begin{enumerate}
    \item $\wh K$ is an irreducible representation of $\cO_\infty$;
    \item If $K,K'$ are two representations of $\cP_\infty$ satisfying the assumptions of above, then for any choices of unit vectors $\Omega,\Omega'$ we have that $\wh K\simeq \wh K'$ (as representations of $\cO_\infty$) if and only if $K\simeq K'$ (as representations of $\cP_\infty$).
\end{enumerate}
\end{theorem}

\begin{proof}
We start by describing $\wh K$ in a handy manner.
Consider the set $\N^*$ of words over the natural numbers $\N$ and the Hilbert space $\ell^2(\N^*)\otimes K$ that we mod out by the equivalence relation generated by
$$\delta_\mu\otimes \xi \sim \sum_{j\in \N} \delta_{\mu\cdot j} \otimes A_j \xi.$$
Hence, $\ell^2(\N^*)$ corresponds to the Fock space of $\ell^2(\N)$ with orthonormal basis $\delta_\mu$ indexed by words $\mu=\mu_1\cdots\mu_k$ with letters $\mu_i$ in $\N$.
Write $H$ for the quotient space of above and write $[\mu,\xi]$ for the image of $\delta_\mu\otimes\xi$ in $H$.
We define an inner product on it as follows:
$$\langle [\mu,\xi],[\nu,\eta]\rangle=0$$
if $\mu$ is not an initial subword of $\nu$ and if $\nu$ is not an initial subword of $\mu$.
Take now $[\mu\gamma,\xi]$ and $[\mu,\eta]$ where $\mu,\gamma$ are words.
We set $$\langle [\mu \gamma,\xi], [\mu,\eta]\rangle :=\langle \xi, A_\gamma \eta\rangle.$$
This defines in a unique manner a sesquilinear map on $H$ which is an inner product.

Indeed, it is equal to the following direct limit of Hilbert spaces.
For each $k\geq 1$ we set $H_k$ to be the tensor product $\ell^2(\N^k)\otimes K$: the $\ell^2$-functions $\xi:\N^k\to K, \mu\mapsto \xi(\mu).$
Now, consider the connecting map
$$\iota_k:H_k\to H_{k+1},\ \iota_k(\xi)(\mu\cdot j):=A_j(\xi(\mu))$$
for $k\geq 1, j\in\N, \xi\in H_k,\mu\in \N^k.$
This is an isometry precisely because we are under the assumption that $$\sum_i A_i^*A_i=\id_K$$
(i.e.~the net of finite sums converges in norm to $\id_K$).
One can now easily check that $H$ and $\varinjlim_n H_n$ are unitary equivalent under the obvious map.

We claim that the formula
$$T:S_\mu S_\nu^*\wh\Omega\mapsto [\mu, A_\nu \Omega]$$
defines a unitary transformation from $\wh K$ to $H$.

To see this recall that the set of vectors $s_\mu s_\nu^*$ where $\mu,\nu$ are finite words in $\N$ is total inside $\cO_\infty$.
Hence, the set of vectors $S_\mu S_\nu^*\wh\Omega$ is total inside the Hilbert space $\wh K$.
Take four words $\mu,\nu,\alpha,\beta$ and note that
$$\langle S_\mu S_\nu^*\wh\Omega,S_\alpha S_\beta^*\wh\Omega\rangle=f\circ\varphi_\infty(s_\beta s_\alpha^* s_\mu s_\nu^*)$$
where $$f(x)=\langle x\cdot \Omega,\Omega\rangle,\ x\in\cP_\infty$$
is the vector state associated to $(A_i,K,\Omega).$
When $\mu$ is not an initial subword of $\alpha$ and $\alpha$ is not an initial subword of $\mu$ we obtain zero.
Assume now that $\alpha=\mu\gamma$ for some word $\gamma.$
We obtain:
$$\langle S_\mu S_\nu^*\Omega,S_\alpha S_\beta^*\Omega\rangle=f\circ\varphi_\infty(s_\beta s_\alpha^*s_\mu s_\nu^*)=f\circ \varphi_\infty(s_\beta s_\gamma^* s_\nu^*) = f(a_\beta^* a_\gamma a_\nu) =\langle A_\gamma A_\nu \Omega, A_\beta \Omega\rangle.$$
This latter is equal to 
$$\langle [\mu, A_\nu\Omega], [\alpha, A_\beta\Omega]\rangle.$$
This proves that the formula $S_\mu S_\nu^*\wh\Omega\mapsto [\mu, A_\nu\Omega]$ provides a well-defined linear map $T$ whose domain is a dense subspace of $\wh K$ and that is isometric. 
By density it extends uniquely to an isometry.
By assumption the closed linear span of the $A_\mu\Omega$ is equal to $K$ (where note we do not consider adjoints of the generators $A_i$'s).
This implies that $T$ has dense range. Since $T$ is isometric it has complete range and thus $T$ is surjective.
We have proven that $T$ is a unitary transformation.

Observe now that by transporting the action of $\cO_\infty$ on $H$ we obtain the formula:
$$S_\nu [\mu,\eta]=[\nu\mu,\eta].$$
In particular, $S_\nu^*[\mu,\eta]$ is zero unless $\mu\in \nu\N^*$ or $\nu\in \mu\N^*.$
Moreover, 
\begin{equation}\label{eq:intertwinne}S_\nu^* [\emptyset,\eta]=[\emptyset, A_\nu \eta].\end{equation}
Hence, we have an obvious isometric embedding $$K\into \wh K, \eta\mapsto [\emptyset,\eta]$$
which intertwines the $A_i$'s with the adjoint of the $S_i$'s.
From now on we identify $\wh K$ with $H$ and identify $K$ with its image inside $H$.
Hence, we write $\eta$ rather than $[\emptyset,\eta]$ if $\eta\in K.$

Our aim is to show that $\wh K$ remembers $K$ in the following manner.
We consider Hilbert subspaces of $\hat K$ that are closed under the actions of the $S_i^*$'s (but they could not be closed under the actions of the $S_j$'s without the star).
We call such subspaces \emph{submodules} of $\wh K.$
A specific submodule is (the image of) $K$. Indeed, we have the formula $S_i^*\eta=A_i\eta$ for $\eta\in K$ and thus $K$ is closed under the action of the $S_i^*$.
We will show that all nonzero submodules of $\wh K$ contain $K$.

Let $\xi\in \wh K$ be a unit vector and write $M_\xi$ for the closed linear span of the $S_\mu^*\xi$, i.e.~the submodule generated by $\xi$.
Fix $0<\varep<1/4.$
By density there exists a unit vector $\eta$ of the form $\sum_{\mu\in I}[\mu,\eta_\mu]$ with $I$ finite so that $\|\xi-\eta\|<\varep.$
Up to using the formula $[\mu,\zeta]=\sum_{j\in\N} [\mu\cdot j, A_j\zeta]$ we may moreover assume that for all $\mu\neq\mu'$ inside $I$ we have that $\mu\notin \mu'\N^*$ and $\mu'\notin \mu\N^*.$
For instance take all the $\mu$'s in $I$ of same length say $\ell$.
This assures that the $[\mu,\eta_\mu]$ are pairwise orthogonal for $\mu\in I.$
Moreover, the $S_\mu$'s with $\mu\in I$ have pairwise orthogonal ranges and thus $$\|\sum_{\mu\in I} S_\mu S_\mu^* \xi\|^2=\sum_{\mu\in I} \|S_\mu^*\xi\|^2.$$
Up to growing $I$ (i.e.~adding words of length $\ell$) we have that $\xi$ and $\sum_{\mu\in I} S_\mu S_\mu^* \xi$ can be made arbitrarily close, say $\varep$-close.
We deduce that 
\begin{align*}
\sum_{\mu\in I} \|S_\mu^*\xi - \eta_\mu\|^2 & = \sum_{\mu\in I} \|S_\mu S_\mu^*\xi - [\mu,\eta_\mu]\|^2\\
& =\| \sum_{\mu\in I} S_\mu S_\mu^* \xi - \eta\|^2\\
& \leq ( \|\sum_{\mu\in I} S_\mu S_\mu^* \xi - \xi\| + \|\xi-\eta\|)^2 < 4 \varep^2.  
\end{align*}
Write $J\subset I$ for the subset of $\mu\in I$ satisfying $S_\mu^*\xi\neq 0$ and $\eta_\mu\neq 0.$

We claim that there is at least one $\mu\in J$ satisfying $$\| \dfrac{S_\mu^*\xi}{\|S_\mu^*\xi\|} - \dfrac{\eta_\mu}{\|\eta_\mu\|}\|<4\varep.$$

Indeed, observe that there is at least one $\mu\in I$ satisfying $$\|S_\mu^*\xi-\eta_\mu\|<2\varep\cdot \|\eta_\mu\|$$ since otherwise we would obtain
$$\sum_{\mu\in I}\|S_\mu^*\xi - \eta_\mu\|^2   \geq \sum_{\mu\in J} (2\varep)^2\cdot \|\eta_\mu\|^2 =4\varep^2$$
contradicting a previous inequality.
Take this specific $\mu$. It must be in $J$.
We have
\begin{align*}
\| \dfrac{\eta_\mu}{\|\eta_\mu\|}-\dfrac{S_\mu^*\xi}{\|S_\mu^*\xi\|} \| & \leq \| \dfrac{\eta_\mu}{\|\eta_\mu\|}-\dfrac{S_\mu^*\xi}{\|\eta_\mu\|} \| + \| \dfrac{S_\mu^*\xi}{\|\eta_\mu\|}-\dfrac{S_\mu^*\xi}{\|S_\mu^*\xi\|} \|\\
& < 2 \varep + \| \dfrac{S_\mu^*\xi}{\|\eta_\mu\|}-\dfrac{S_\mu^*\xi}{\|S_\mu^*\xi\|} \|\\
& = 2 \varep + | \dfrac{\|S_\mu^*\xi\|}{\|\eta_\mu\|}- 1 |\\
& = 2 \varep + | \dfrac{\|S_\mu^*\xi\| - \|\eta_\mu\|}{\|\eta_\mu\|} |\\
& \leq 2 \varep + \dfrac{\|S_\mu^*\xi - \eta_\mu\|}{\|\eta_\mu\|} \\
& < 4 \varep.
\end{align*}

We have proven that the distance between the unit spheres of $M_\xi$ and $K$ is smaller than $4\varep$ and thus must be zero.
We now use the assumption 
$$\dim(K)<\infty.$$
Indeed, this implies that the unit sphere of $K$ is compact.
Therefore, there must exist a vector $\zeta$ in the unit sphere of $K$ at distance zero from the unit sphere of $M_\xi$.
This means that $\zeta$ is in $M_\xi$ and thus $K\cap M_\xi$ is nonzero.
We now use the strong irreduciblity assumption: the space $K$ is irreducible as module (i.e.~it does not contain any proper nonzero Hilbert subspace closed under the action of the adjoints of the $S_i$'s). 
Hence, $K\subset M_\xi$.
Therefore, any nonzero submodule of $\wh K$ contains $K$.
We may then canonically recover $K$ inside $\wh K$ by taking the intersection of all nonzero submodules of $\wh K.$

Assume now that $H_0\subset \wh K$ is a subrepresentation of $\cO_\infty$.
The space $H_0$ is in particular closed under the action of the $S_i^*$ and is thus a submodule.
Hence, it contains $K$ and thus contains the vector $\Omega$. 
This vector is clearly cyclic for the representation $\wh K$ of $\cO_\infty$ since $S_\mu S_\nu^*\Omega=[\mu,A_\nu\Omega]$ and since $\{A_\nu\Omega:\ \nu\in\N^*\}$ is total inside $K$.
Hence, $H_0=\wh K$ and $\wh K$ is an irreducible representation of $\cO_\infty.$

Assume now that we have two representations $K,K'$ of $\cP_\infty$ satisfying the assumption of the theorem.
Assume that there is a unitary transformation $u:\wh K\to \wh K'$ intertwining the $\cO_\infty$-actions.
In particular, it intertwines the module structures. 
It must then send $K$ to $K'$ since $K$ (resp.~$K'$) is equal to the intersection of all nonzero submodules of $\wh K$ (resp.~$\wh K'$).
Hence, the map $u$ realises an isomorphism of modules from $K$ to $K'$: $u(A_i\eta)=A_i'u(\eta)$ for all $\eta\in K,i\geq 1$ and $u(K)=K'.$
Since $u$ is a unitary we deduce that $u$ restricted from $K$ to $K'$ intertwines the $A_i^*$ with the $(A_i')^*$. Therefore, $K$ and $K'$ are isomorphic as representations of $\cP_\infty$.
\end{proof}

By specialising the previous theorem to the dimension $1$ case we deduce the following.

\begin{corollary}
Let $z=(z_i)_{i\geq 1}$ be a point of the unit sphere $\cS_\infty$ of $\ell^2(\N)$, i.e.~$\sum_i|z_i|^2=1$ and $z_i\in\C.$
The map $a_i\mapsto z_i$ defines a one-dimensional representation $\pi_z$ of $\cP_\infty$. 
Denote by $H_z$ the associated representation of $\cO_\infty$ obtained from the GNS triple $(z,\C,1).$
The representations $H_z$ with $z\in \cS_\infty$ are all irreducible and pairwise inequivalent. 
\end{corollary}

\end{document}